\documentclass[a4paper,12pt,reqno]{amsart}
\usepackage{latexsym}
\usepackage{amssymb} 
\usepackage{mathrsfs}
\usepackage{amsmath}
\usepackage{latexsym}

\usepackage{delarray}
\usepackage{amssymb,amsmath,amsfonts,amsthm,mathrsfs}
\usepackage{blkarray}

\usepackage{hyperref}
\renewcommand\eqref[1]{(\ref{#1})} 
 \newtheorem{thm}{Theorem}[section]
 \newtheorem{cor}[thm]{Corollary}
 \newtheorem{lem}[thm]{Lemma}
 \newtheorem{prop}[thm]{Proposition}
 \theoremstyle{definition}
 \newtheorem{defn}[thm]{Definition}
 \theoremstyle{remark}
 \newtheorem{rem}[thm]{Remark}
 \newtheorem{ex}[thm]{Example}
 \numberwithin{equation}{section}

\newcommand{\half}{\frac{1}{2}}

\newcommand{\ene}{\mathbb{N}}

\newcommand{\LL}{\mathcal{L}}

\newcommand{\er}{\mathbb{R}}
\newcommand{\ar}{\mathbb{R}}
\newcommand{\ce}{\mathbb{C}}
\newcommand{\zet}{\mathbb{Z}}

\newcommand{\T}{\mathbb{T}^1}

\newcommand{\fou}{\mathcal{F}}

\newcommand{\bi}{\begin{itemize}}

\newcommand{\ei}{\end{itemize}}
\newcommand{\be}{\begin{enumerate}}
\newcommand{\ee}{\end{enumerate}}
\newcommand{\beq}{\begin{equation}}
\newcommand{\eq}{\end{equation}}

\def\Op{{{\rm Op}}}

\DeclareMathOperator{\Tr}{Tr}
\DeclareMathOperator{\Det}{Det}

\def\Rn{{{\mathbb R}^n}}
\def\Tn{{{\mathbb T}^n}}
\def\Zn{{{\mathbb Z}^n}}
\def\T{{{\mathbb T}^1}}
\def\N{{{\mathbb N}}}
\def\C{{{\mathbb C}}}
\def\SU2{{{\rm SU(2)}}}
\def\SO3{{{\rm SO(3)}}}
\def\lapsu2{{{\mathcal L}_\SU2}}
\def\lapsu2{{{\mathcal L}_\SU2}}

\def\Re{{{\rm Re}\,}}
\def\Op{\text{\rm Op}}

\DeclareMathOperator{\Hi}{{\bf H}}
\DeclareMathOperator{\Hcal}{\bf H}

\begin{document} 

%
%
%
%
%
%
%
%
%
\title[Schatten-von Neumann classes of tensors ]
 {Schatten-von Neumann classes of  tensors of invariant operators }

\author{Julio Delgado}

\address{Universidad del Valle\\
Departamento de Matematicas\\
Calle 13 100-00\\
Cali\\
Colombia}

\email{delgado.julio@correounivalle.edu.co}

\author{Liliana Posada}

\address{Universidad del Valle\\
Departamento de Matematicas\\
Calle 13 100-00\\
Cali\\
Colombia}

\email{liliana.posada@correounivalle.edu.co}

\author[Michael Ruzhansky]{Michael Ruzhansky}

\address{%
		Ghent University\\
	Department of Mathematics: Analysis, Logic and Discrete Mathematics\\
	Krijgslaan 281, Building S8 \\
	B 9000 Ghent\\
	Belgium}

\email{Michael.Ruzhansky@ugent.be}

\address{%
	Queen Mary University of London
	School of Mathematical Sciences\\
	Mile End Road\\
	London E1 4NS\\
	United Kingdom}
\email{m.ruzhansky@qmul.ac.uk}

\thanks{The authors are supported by the FWO Odysseus 1 grant G.0H94.18N: Analysis and Partial Differential Equations and by the Methusalem programme of the Ghent University Special Research Fund (BOF) (Grant number 01M01021). Michael Ruzhansky is also supported by FWO Senior Research Grant G011522N and EPSRC grants EP/R003025/2 and EP/V005529. The first and second authors were also supported by Vic Inv Universidad del Valle CI-71352.}



\date{\today}
\begin{abstract}
In this work we study Schatten-von Neumann classes of tensor products of invariant operators on  Hilbert spaces. In the first part we  first deduce some spectral properties for tensors of anharmonic oscillators thanks to the knowledge on corresponding Schatten-von Neumann properties. In the second part we specialised on tensors of invariant operators. In the special case where a suitable Fourier analysis  associated to a fixed partition of a Hilbert space into finite dimensional subspaces  is available we also give the corresponding formulae in terms of symbols. We also give a sufficient condition for Dixmier traceability for a class of finite tensors of pseudo-differential operators on the flat torus. 
\end{abstract}

\maketitle

\section{Introduction}
The notion of finite tensors of Hilbert spaces appears for the first time in the works of  Murray and von Neumann in \cite{Murray}. The more general case of infinite tensors was extended by von Neumann in \cite{Ne}.  A recent account on these tensors can be found in  \cite{weav:book}.\\

The trace class is a fundamental concept in quantum mechanics. A  {\em  density operator} or a {\em statistical operator} is a  positive self-adjoint trace class operator for a corresponding ensemble of quantum systems, thus a special element of $S_1(\Hi)$. On the other hand, when systems interact the space of states takes the form of a tensor of Hilbert space, that is, a tensor of spaces of states.  The Schatten-von Neumann class  $S_2(\Hi)$  is also known as the class of Hilbert-Schmidt operators. In the physics  literature, $S_2(\Hi)$ it is known as the {\em Liouville space} over $\Hi$. When dealing with coupled quantum systems or open quantum systems, the tensor products of Hilbert spaces arise. More specifically, if $\Hi_1, \Hi_2 $ are  Hilbert spaces and the state spaces of respective quantum systems, then the Hilbert tensor product  $\Hi_1\otimes\Hi_2$ is the state space of the corresponding coupled system. Hilbert tensor products of more factors are required  when the number of systems interacting increases.   An important application of Hilbert tensor products and Schatten-von Neumann classes naturally arises in the study of quantum channels in Quantum Information Theory. The use of Hilbert tensors in Quantum Mechanics goes back to the beginnings of its  mathematical foundations.  \\

The Schatten-von Neumann norms have been recently applied for measurements of quantum statistical speed. In \cite{mgas:sp}, one introduced the notion 
 of Schatten-von Neumann speed.  Therein, each Schatten-von Neumann norm is used to define a corresponding quantum  statistical distance which in turn induces a quantum statistical speed. The measurement of a quantum statistical speed quantifies the sensitivity of an initial state with respect to changes of the parameter of a dynamical evolution. The statistical speed of a quantum state can be interpreted as an observable witness for entanglement. Further applications of Schatten-von Neumann speeds can be found in \cite{schit:cc1}.\\


We will apply the Fourier analysis associated to a fixed partition of a Hilbert space into finite dimensional subspaces as introduced in \cite{fjdmr:foum}. From it one can define a notion of global symbol corresponding to a Fourier multiplier. We then consider tensors of Fourier multipliers and derive sharp Schatten-von Neumann properties and trace formulae  on Hilbert tensor products in terms of global symbols. We will also study some consequences for pseudo-differential operators on the torus that can be decomposed as a composition of a multiplication operator and a invariant operator(Fourier multiplier).\\

The authors have investigated different Schatten-von Neumann properties in previous works. Sharp sufficient conditions for kernels of integral operators on compact manifolds \cite{dr:suffkernel} and in other settings \cite {dr:intsc}, anharmonic oscillators \cite{anh:cdr},  Schatten-von Neumann classes and trace formulae for Fourier multipliers  on compact groups  \cite{dr13:schatten}, Grothendieck-Lidskii formulas and nuclearity in \cite{dr13a:nuclp}, Schatten-von Neumann properties for Fourier multipliers associated to partitions on Hilbert spaces \cite{fjdmr:foum}. The Dixmier trace and Wodzicki residue have been investigated for pseudo-differential oeprators on compact manifolds and compact Lie groups in \cite{cdc20}, \cite{ckc20}, \cite{cc20}.\\

In Section \ref{sec2} we review some basic definitions on Schatten-von Neumann Ideals and finite tensor products of Hilbert spaces. Section \ref{sec3} is devoted to the case of Schatten-von Neumann properties on finite  tensor products and the special case of the trace class. We deduce spectral properties for tensors of anharmonic oscillators and in particular for the rate of growth of the energy levels of the tensor.  In Section  \ref{sec4} we review the notion of global symbol for  a class of invariant operators with respect to a partition of a Hilbert space into finite dimensional subspaces in the sense of \cite{fjdmr:foum}. Therein we use the notion of the Fourier analysis associated to that kind  of partitions. In Section \ref{sec5}, we apply the notion of global symbol to the setting of  tensors of invariant operators.  We obtain some Schatten-von Neumann properties for tensors of invariant operators on compact Lie groups. At the end of the section we obtain a sufficient condition for Dixmier traceability for a class of finite tensors of pseudo-differential operators on the flat torus with symbols of the form $\sigma(x,j)=a(x)\beta(j)$ defined on $\Tn\times \Zn$. This class of symbols enjoy some special features for the analysis and have been the object of recent interest in the study under the form of the so-called  pseudo-differential neural operators  (cf. \cite{pnetw:ko}).  


 
\section{Tensor product of Hilbert spaces and Schatten-von Neumann ideals}\label{sec2}
In this section we recall the basic elements of  tensor products of Hilbert spaces and Schatten-von Neumann ideals. \\

Let $\Hi$ be a complex Hilbert space, we denote its inner product by $\langle \cdot,\cdot\rangle_{\Hi}$ or simply by $\langle \cdot,\cdot\rangle$ when there is not room for confusion. Since the main applications related to this work arise from  quantum mechanics, we adopt here the convention on the antilinearity on the first factor (or conjugate homogeneity) of the inner product, i.e., 
\[ \langle a\phi,\psi\rangle=\overline{a}\langle \phi,\psi\rangle,\]
 for all $\phi,\psi\in\Hi, a\in\ce.$\\

We will sometimes also use the Dirac's {\em Bra} and {\em Ket} notation. If $\psi$ is an element of the Hilbert space  $\Hi$ we will say that it is a {\em ket} and write $|\psi\rangle$. If  $\phi$ is an element of the Hilbert space  $\Hi$ we will associate to it an element of $\Hi^*$ that we call the Bra-$\phi$ denoted by  $\langle \phi|$. The  Bra-$\phi$ defines a
 linear form on $\Hi$ given by 
\[\langle \phi|(|\psi\rangle):=\langle \phi,\psi\rangle.\]
The notation can be simplified and justified by the Riesz representation Theorem writing
\[\langle \phi|\psi\rangle=\langle \phi,\psi\rangle.\]

The trace class is a crucial object to define several important concepts in the mathematical framework of quantum mechanics, and in particular it is instrumental  for the main applications that  will be considered herein. We now recall its definition. \\

Let $T:\Hi\rightarrow \Hi$ be an operator in $S_1(\Hi)$ and let  $\{\phi_k\}_k$ be any orthonormal basis for the Hilbert space $\Hi$. Then, the series $\sum\limits_{k=1}^{\infty}\langle \phi_k,T\phi_k\rangle_{\Hi}$ is absolutely convergent and the sum is independent of the choice of the orthonormal basis $\{\phi_k\}_k$. Thus, we can define the trace $\Tr (T)$ of any linear operator
$T:\Hi\rightarrow \Hi$ in $S_1(\Hi)$ by $$\Tr (T)=\sum\limits_{k=1}^{\infty}\langle \phi_k,T\phi_k\rangle_{\Hi},$$
where $\{\phi_k:k=1,2,\dots\}$ is any orthonormal basis for $\Hi$. \\


\subsection{Tensor products of Hilbert spaces and operators}
We shall now recall the definition and basic properties of tensor products of Hilbert spaces and operators which are key concepts in Quantum Mechanics. They arise when several quantum systems are involved and interact with each other. 

\begin{defn}
	Let $\Hi_1, \Hi_2$ be  two complex  Hilbert spaces. Let $\phi\in \Hi_1$ and $\psi\in \Hi_2$. We define 
	 $\phi\otimes\psi$ to be the bi-antilinear form on $ \Hi_1\times \Hi_2$ given by 
	 \[\phi\otimes\psi(\phi',\psi')=\langle \phi',\phi \rangle\langle \psi',\psi \rangle ,\]%


\vspace{0.5cm}

	Let $\mathcal{E}$ be the linear span of terms of the form $\phi\otimes\psi$, and we will denote it by $ \Hi_1\widetilde{\otimes} \Hi_2$. The vector space $\mathcal{E}=\Hi_1\widetilde{\otimes} \Hi_2$ can be endowed with the inner product  defined by 
	\begin{align*}\langle \phi\otimes\psi, \phi'\otimes\psi'\rangle&=\langle \phi,\phi' \rangle\langle \psi,\psi' \rangle\\
	&= \phi'\otimes\psi'(\phi,\psi).
		\end{align*}
  
In fact, if $u=\sum\limits_{i=1}^n \phi_i \otimes \psi_i \in \mathcal{E}$, then the norm $u$ induced for the product $\langle \cdot, \cdot\rangle$ is
\[\left\|\sum\limits_{i=1}^n \phi_i \otimes \psi_i\right\|=\left\{\sum\limits_{i=1}^n\sum\limits_{k=1}^n \langle \phi_i,\phi_k \rangle\langle \psi_i,\psi_k \rangle\right\}^{1/2}.\]
 
\end{defn}
One can show that the product $\langle \cdot, \cdot\rangle$ on $\mathcal{E}$ is well defined and positive definite.
 The space $(\mathcal{E}, \langle \cdot, \cdot\rangle)$ is then a pre-Hilbert space. 
  It is clear that we can extend the above definition of tensor product  to the case of a finite number of Hilbert spaces $\Hi_1, \Hi_2,\dots, \Hi_n$; getting again a  corresponding pre-Hilbert space  $(\mathcal{E}, \langle \cdot, \cdot\rangle)$.    
  The notion of Hilbert tensors was introduced by Murray and von Neumann in \cite{Murray} for finite tensors.

\begin{defn} The completion of $(\mathcal{E}, \langle \cdot, \cdot\rangle)$ is denoted by  $ \Hi_1\otimes\cdots \otimes \Hi_n$. 
	  The space $\Hi_1\otimes\cdots \otimes \Hi_n$ is called the {\em tensor product} of the  complex  Hilbert spaces $\Hi_1,\dots , \Hi_n$.
\end{defn}

We will denote by $\ene$ the set of natural numbers $\{1,2,\dots\}$ and by $\ene_0$ the set of natural numbers including $0$, i.e. $\ene_0 =\{0,1,2,\dots \}.$ The following theorem is well-known and is a basic property for the construction of a basis in tensor products.  
\begin{thm}\label{prhdes} Let $(e_k)_{k\in \mathcal{N}}$, $(f_{\ell})_{\ell\in \mathcal{M}}$ be orthonormal bases of the complex Hilbert spaces $\Hi_1$ and $\Hi_2$ respectively. Then, the family $(e_j\otimes f_{\ell})_{(k,\ell)\in\mathcal{N}\times \mathcal{M}}$ is an orthonormal basis of $\Hi_1\otimes\Hi_2$.
\end{thm}
It is clear that the previous theorem can be extended to finitely many complex Hilbert spaces  $\Hi_1,\dots , \Hi_n$.
\begin{cor}\label{prhdes33} Let $\Hi_1,\dots , \Hi_n$ be complex Hilbert spaces. Let $(e_k^i)_{k\in \mathcal{N}_i}$ be an orthonormal basis of each $\Hi_i$,  $i=1,2,\dots , n$. Then, $(e_{k_1}^1\otimes\cdots \otimes  e_{k_n}^n)_{(k_1,\dots ,k _n)\in\mathcal{N}_1\times\cdots \mathcal{N}_n}$ is an orthonormal basis of  $ \Hi_1\otimes\cdots \otimes \Hi_n$.
\end{cor}

We observe that the Hilbert spaces are not assumed separable in the previous statements. Non-separable Hilbert spaces arise in several situations,  in particular the countable tensor product of separable Hilbert spaces is not separable.  	
See for instance  \cite{gg:qq1}, \cite{gg:qq2} for some recent model proposals in the context of Quantum Field Theory within the non-separable setting. Hereafter we will only consider complex separable Hilbert spaces.
\medskip%

 We refer the reader to \cite{weid:b1} for a detailed treatment on the basics of Schatten-von Neumann classes of operators between different Hilbert spaces. Regarding the Hilbert tensor products
 the reader can refer to \cite{r-s:vol1}, \cite{prug:b1} for more comprehensive treatments on these topics. However we should point out that some preliminaries regarding tensor products of Hilbert spaces and specially partial traces are not easily accessible from the existing  literature. In particular we have included some basics on countable tensor products of Hilbert spaces.  
 In this regard we refer to \cite{gui:hi}, \cite{Pah:book} for more detailed discussions for some of those properties. 
\subsection{Schatten-von Neumann classess and the trace}
We now recall the notion of Schatten-von Neumann ideals. Let  $\Hi_1,\Hi_2$ be  complex Hilbert spaces. We denote by $\mathcal{L}(\Hi_1, \Hi_2)$ the algebra of bounded linear operators from $\Hi_1$ into $\Hi_2$ endowed with the operator norm. If $\Hi_1=\Hi_2$ we simply write $\mathcal{L}(\Hi)$. A linear compact operator $A:\Hi_1\rightarrow \Hi_2$ is said to belong to the Schatten-von Neumann class $S_r(\Hi_1,\Hi_2)$ for $0<r<\infty$ if 
$$\sum\limits_{n=1}^{\infty}(s_n(A))^r<\infty,$$ 
where $s_n(A)$ denote the singular values of $A$, i.e. the eigenvalues of $|A|=\sqrt{A^*A}$,  
with multiplicities counted. If $1\leq r<\infty$ 
the class $S_r(\Hi_1,\Hi_2)$ becomes a Banach space endowed with the norm 
\[\|A\|_{S_r}=\left(\sum\limits_{n=1}^{\infty}(s_n(A))^r\right)^{\frac{1}{r}}.\]
If $0<r<1$ the identity above only defines a quasi-norm with respect to which $S_r(\Hi_1,\Hi_2)$ is complete. In  the case $\Hi=\Hi_1=\Hi_2$ we simply denote $S_r(\Hi_1,\Hi_2)=S_r(\Hi)$. The class $S_2(\Hi_1,\Hi_2)$ and $S_1(\Hi)$ are usually known as the class of Hilbert-Schmidt operators and the trace class, respectively. In the case of $r=\infty$ we can put $\|A\|_{S_{\infty}}$ to be the operator norm of the bounded operator
$A:\Hi_1\to \Hi_2$. In this case  $S_{\infty}(\Hi_1,\Hi_2)$ is the class of compact operators from 
 $\Hi_1$ into $\Hi_2$ endowed with the operator norm.\\

\section{Schatten-von Neumann properties of Tensors }\label{sec3}
 We start by recalling the definition of tensors of operators and deducing some basic properties on tensors of Schatten-von Neumann operators.

\subsection{Tensor products of operators}
We can now define the tensor product of operators. 
\begin{defn} Let $A$ be an operator on $ \Hi_1$, with a dense domain  $Dom A$, and $B$ be a operator on $ \Hi_2$ with dense domain $Dom B$. Let $D$ denote the space  $Dom A\overline{\otimes} Dom B$, that is the space of finite linear combinations of $\phi\otimes\psi$ with $\phi\in Dom A$ and $\psi\in Dom B$. Clearly $D$ is dense in  $ \Hi_1\otimes \Hi_2$. One defines the operator $A\otimes B$ on  $D$ by 
\[(A\otimes B)(\phi\otimes\psi)=A\phi\otimes B\psi\]
and its linear extension to all $D$.
	\end{defn}
It is clear that the definition above can be natually extended to a finite family of densely defined operators.
\begin{prop}
	The operator $A\otimes B$ is well defined. If $A$ and $B$ are closable operators then so is $A\otimes B$. 
\end{prop}
\begin{defn}
	If $A$ and $B$ are closable operators, we call {\em tensor product } of  $A$ by $B$ the closure of $A\otimes B$ and it still denoted by $A\otimes B$. 
\end{defn}
The following proposition follows immediately from the definition of the tensor $A\otimes B$.
\begin{prop}\label{proboundt} If $A$ and $B$ are bounded operators  on $ \Hi_1$ and  $ \Hi_2$, respectively, then $A\otimes B$ is a bounded operator on  $ \Hi_1\otimes \Hi_2$ and we have
	\[\|A\otimes B\|=\|A\|\|B\|.\] 
\end{prop}

A similar property holds for Schatten-von Neumann classes. We also give some addittional  properties regarding the trace of tensors in the following theorem. We recall their proofs for the convenience of the reader since they are not easy to find in the literature. 

\begin{thm}\label{scprt} Let $\Hi_1, \Hi_2$ be  two complex separable Hilbert spaces. Let $A$ and $B$ be bounded operators  on $ \Hi_1$ and  $ \Hi_2$, respectively.
	\begin{enumerate}
		\item If $A$ and $B$ are compact  operators, then $A\otimes B$ is a compact operator on  $ \Hi_1\otimes \Hi_2$. 

\item Let $0<p<\infty$. If $A\in S_p(\Hi_1)$ and $B\in S_p(\Hi_2)$. Then  $A\otimes B\in S_p( \Hi_1\otimes \Hi_2)$. Moreover we have
\[\|A\otimes B\|_p=\|A\|_p\|B\|_p,\]
and for $p=1$, the following formula for the trace holds
	\[\Tr (A\otimes B)=\Tr (A)\Tr (B).\]
\end{enumerate}
\end{thm}
\begin{proof} (1) Since $A$ and $B$ are compact operators we can write
\[A=\sum\limits_{i=1}^{\infty}\alpha_i|u_i\rangle\langle v_i|,\,\, B=\sum\limits_{j=1}^{\infty}\beta_j|t_j\rangle\langle w_j|, \]
where $(u_i)_i, (v_i)_i $ and $(t_j)_j, (w_j)_j $ are orthonormal sequences in $ \Hi_1$ and  $ \Hi_2$ respectively, and $(\alpha_i)_i, (\beta_j)_j$ are real non-negative sequences converging to $0$.\\

Then, we can write 
\beq A\otimes B=\sum\limits_{i,j=1}^{\infty}\alpha_i\beta_j|u_i\otimes t_j\rangle\langle v_i\otimes w_j|.\label{trAtB}\eq
Hence we obtain a polar representation of $A\otimes B$:
\beq A\otimes B=\sum\limits_{i,j=1}^{\infty}\alpha_i\beta_j|v_i\otimes w_j\rangle\langle v_i\otimes w_j|.\label{tenab}\eq
The compactness now follows immediately since $ (v_i\otimes w_j)_{ij}$ is an orthonormal sequence in $ \Hi_1\otimes \Hi_2$ and $(\alpha_i\beta_j)_{ij}$ converges to $0$ as $(i,j)$ goes to $\infty$.\\

(2) By \eqref{tenab} with  $A\in S_p(\Hi_1), B\in S_p(\Hi_2)$, we note that
\[\sum\limits_{i,j=1}^{\infty}(\alpha_i\beta_j)^p=\sum\limits_{i=1}^{\infty}\alpha_i^p\sum\limits_{j=1}^{\infty}\beta_j^p=\|A\|_p^p\|B\|_p^p .\]
Therefore  $A\otimes B\in S_p( \Hi_1\otimes \Hi_2)$ and 
\[\|A\otimes B\|_p=\|A\|_p\|B\|_p.\]

For $p=1$, in order to obtain the formula for the trace, we first note that there exist orthonormal basis  $(u_i)_i, (v_i)_i $ and $(t_j)_j, (w_j)_j $ of $ \Hi_1$ and  $ \Hi_2$ respectively, and complex sequences $(\alpha_i)_i, (\beta_j)_j$ converging to $0$ such that
\[A=\sum\limits_{i=1}^{\infty}\alpha_i|u_i\rangle\langle v_i|\,,\,\,\, B=\sum\limits_{j=1}^{\infty}\beta_j|t_j\rangle\langle w_j|. \]
Hence
\beq A\otimes B=\sum\limits_{i,j=1}^{\infty}\alpha_i\beta_j|u_i\otimes t_j\rangle\langle v_i\otimes w_j|.\label{trAtB2}\eq

Since  $ (v_i\otimes w_j)_{ij}$ is an orthonormal basis of $ \Hi_1\otimes \Hi_2$,  by \eqref{trAtB2} we obtain
 \[ (A\otimes B)(v_k\otimes w_l)=\alpha_k\beta_l|u_k\otimes t_l\rangle\]
 
Therefore
 \begin{align*}
 \Tr(A\otimes B)&=\sum\limits_{i,j=1}^{\infty}\langle v_i\otimes w_j, (A\otimes B)(v_i\otimes w_j)\rangle\\ 
 &=\sum\limits_{i,j=1}^{\infty}\langle v_i\otimes w_j, \alpha_i\beta_j(u_i\otimes t_j)\rangle\\
 &=\sum\limits_{i,j=1}^{\infty}\alpha_i\beta_j\langle v_i\otimes w_j, u_i\otimes t_j\rangle\\
  &=\sum\limits_{i,j=1}^{\infty}\alpha_i\beta_j\langle v_i, u_i\rangle\langle w_j,  t_j\rangle\\
  &=\sum\limits_{i=1}^{\infty}\alpha_i\langle v_i, u_i\rangle
  \sum\limits_{j=1}^{\infty}\beta_j\langle w_j,  t_j\rangle\\
 &=\Tr(A)\Tr(B).
\end{align*} 
\end{proof}
The above theorem  has immediate consequences for finite tensors of operators. The following corollary follows arguing by induction on Proposition \ref{proboundt} and Theorem \ref{scprt}.

\begin{cor} \label{corsp} Let $\Hi_1,\dots,\Hi_n$  be complex separable Hilbert spaces.\\
(i)  If $A_j\in \mathcal{L}(\Hi_j)$, then $A_1\otimes\cdots \otimes A_n\in \mathcal{L}(\Hi_1\otimes\cdots\otimes \Hi_n)$
 and 
 \[\|\bigotimes \limits_{j=1}^nA_j\|_{\mathcal{L}(\Hi_1\otimes\cdots\otimes \Hi_n)}=\prod\limits_{j=1}^{n}\|A_j\|_{\mathcal{L}(\Hi_j)}. \]

\noindent (ii) Let  $0<p<\infty$. If $A_1,\dots, A_n$ belong to $S_p(\Hi_1), \dots, S_p(\Hi_n)$ respectively, then $A_1\otimes\cdots \otimes A_n\in S_p(\Hi_1\otimes\cdots\otimes \Hi_n)$. Moreover 
\[\|\bigotimes \limits_{j=1}^nA_j\|_{S_p(\Hi_1\otimes\cdots\otimes \Hi_n)}=\prod\limits_{j=1}^{n}\|A_j\|_{S_p(\Hi_j)}.\]

In particular, for $p=1$ the trace class we additionally have 

\[\Tr\left(\bigotimes \limits_{j=1}^nA_j\right)=\prod\limits_{j=1}^{n}\Tr(A_j).\]	
\end{cor}%

As an example we study  a system of anharmonic oscillators by using recent results  obtained in \cite{anh:cdr}. We consider  anharmonic oscillators  on  $\Rn$  of the form
 \beq A_{k,\ell}=(-\Delta)^{\ell}+|x|^{2k} \label{gahw1},\eq
where $k,\ell$ are integers $\geq 1$. In \cite{anh:cdr}, it was established  that  $(A_{k,\ell}+1)^{-\mu}$ belongs to the Schatten-von Neumann class $S_p(L^2(\Rn))$ provided $\mu>\frac{(k+\ell)n}{2k\ell p} $.

\begin{thm} Let $k_j,\ell_j$ be integers $\geq 1$ for $j=1,\dots, N$. We assume
\[\mu_j>\frac{(k_j+\ell_j)n}{2k_j\ell_j p} ,\]
    for  $j=1,\dots, N$. We denote $\widetilde{A}_j:=( A_{k_j,\ell_j}+1)^{-\mu_j}$ for $j=1,\dots, N$; with $A_{k_j,l_j}$ anharmonic oscillators as in \eqref{gahw1}. Then,  $\widetilde{A}_j\in S_p(L^2(\Rn))$ for each  $j=1,\dots, N$ and 
   \beq\label{anhsch1}\bigotimes \limits_{j=1}^N\widetilde{A}_j\in S_p(L^2(\mathbb{R}^{nN})).\eq
Moreover, the eigenvalues $\lambda_m$ of $\bigotimes \limits_{j=1}^N\widetilde{A}_j$ satisfy the following rate of decay
\[
		 \lambda_m=o(m^{-\frac{1}{p}})\,,\quad \text{as} \quad m \rightarrow \infty\, .
		 \]
Consequently, the eigenvalues (Energy levels) $E_m$ of $\bigotimes \limits_{j=1}^N A_j$ have a growth of order  at least 
\[
		m^{\frac{1}{p}}\,,\quad \text{as} \quad m \rightarrow \infty\, .
		 \]
\end{thm}
\begin{proof} The first consequence follows from Corollary 5.4 (b) of \cite{anh:cdr}. The property \eqref{anhsch1} follows from  Corollary \ref{corsp} and again  Corollary 5.4(b) of \cite{anh:cdr} and the fact that $L^2(\mathbb{R}^{nN})$ and $\bigotimes \limits_{j=1}^N L^2(\mathbb{R}^{n})$ are isomorphic.\\

The rate of decay is a consequence of \eqref{anhsch1} and Corollary 5.6 of \cite{anh:cdr}, and the estimate on the rate of growth of the $E_m$ follows
\end{proof}

\section{Global symbols of invariant operators with respect to partitions of Hilbert spaces}\label{sec4}%
In this section we recall some basic elements of the notion of Fourier multipliers or invariant operators on Hilbert spaces and their corresponding global symbols introduced in \cite{fjdmr:foum}. These notions will be crucial for the study of tensors of invariant operators in the last section.\\%

Given a complex separable Hilbert space $\Hi$, we will  consider a fixed partition of $\Hi$ into a direct sum of finite dimensional subspaces,
$$\Hcal=\bigoplus_{j} H_{j}.$$
One can  associate a notion of invariance relative to such  partition as we will see from the Theorem \ref{THM:inv-rem} below established in \cite{fjdmr:foum} as Theorem 2.1. We recall it in detail with the corresponding notations since it is essential for our discussion. We denote by $d_j$ the dimension of  $H_{j}$ and by  $\{e_{j}^{k}\}_{1\leq k\leq d_{j}}$  an orthonormal basis of $ H_{j}$. In particular we can apply this notion of invariance to  the setting when $M$ is a compact manifold without boundary, $\Hcal=L^{2}(M)$ and $\Hcal^{\infty}=C^{\infty}(M)$. 
 In this particular example  we can
 think of $\{e_{j}^{k}\}$ being an orthonormal basis given by eigenfunctions
 of an elliptic operator on $M$, and $d_{j}$ the corresponding 
 multiplicities. This notion of invariance has been recently applied in the setting of control theory of Cauchy problems on Hilbert spaces \cite{CGDR:Hi}.  
 
\begin{thm}\label{THM:inv-rem}
	Let $\Hcal$ be a complex separable Hilbert space and let $\Hcal^{\infty}\subset \Hcal$ be a dense
	linear subspace of $\Hcal$. Let $\{d_{j}\}_{j\in\N_{0}}\subset\N$ and let
	$\{e_{j}^{k}\}_{j\in\N_{0}, 1\leq k\leq d_{j}}$ be an
	orthonormal basis of $\Hcal$ such that
	$e_{j}^{k}\in \Hcal^{\infty}$ for all $j$ and $k$. Let $H_{j}:={\rm span} \{e_{j}^{k}\}_{k=1}^{d_{j}}$,
	and let $P_{j}:\Hcal\to H_{j}$ be the corresponding orthogonal projection.
	For $f\in\Hcal$, we denote $$\widehat{f}(j,k):=(f,e_{j}^{k})_{\Hcal}$$ and let
	$\widehat{f}(j)\in \ce^{d_{j}}$ denote the column of $\widehat{f}(j,k)$, $1\leq k\leq d_{j}.$
	Let $T:\Hcal^{\infty}\to \Hcal$ be a linear operator.
	Then the following
	conditions are equivalent:
	\begin{itemize}
		\item[(A)] For each $j\in\ene_0$, we have $T(H_j)\subset H_j$. 
		\item[(B)] For each $\ell\in\ene_0$ there exists a matrix 
		$\sigma(\ell)\in\ce^{d_{\ell}\times d_{\ell}}$ such that for all $e_j^k$ 
		$$
		\widehat{Te_j^k}(\ell,m)=\sigma(\ell)_{mk}\delta_{j\ell}.
		$$
		\item[(C)]  If in addition, $e_j^k$ are in the domain of $T^*$ for all $j$ and $k$, then 
		for each $\ell\in\ene_0 $ there exists a matrix 
		$\sigma(\ell)\in\ce^{d_{\ell}\times d_{\ell}}$ such that
		\[\widehat{Tf}(\ell)=\sigma(\ell)\widehat{f}(\ell)\]
		for all $f\in\Hcal^{\infty}.$
	\end{itemize}
	
	The matrices $\sigma(\ell)$ in {\rm (B)} and {\rm (C)} coincide.
	
	The equivalent properties {\rm (A)--(C)} follow from the condition 
	\begin{itemize}
		\item[(D)] For each $j\in\ene_0$, we have
		$TP_j=P_jT$ on $\Hcal^{\infty}$.
	\end{itemize}
	If, in addition, $T$ extends to a bounded operator
	$T\in{\mathscr L}(\Hcal)$ then {\rm (D)} is equivalent to {\rm (A)--(C)}.
\end{thm} 

Under the assumptions of Theorem \ref{THM:inv-rem}, we have the direct sum 
decomposition
\begin{equation}\label{EQ:sum}
\Hcal = \bigoplus_{j=0}^{\infty} H_{j},\quad H_{j}={\rm span} \{e_{j}^{k}\}_{k=1}^{d_{j}},
\end{equation}
and we have $d_{j}=\dim H_{j}.$
The two applications that we will consider will be with $\Hcal=L^{2}(M)$ for a
compact manifold $M$ with $H_{j}$ being the eigenspaces of an elliptic 
pseudo-differential operator $E$, or with $\Hcal=L^{2}(G)$ for a compact Lie group
$G$ with $$H_{j}=\textrm{span}\{\xi_{km}\}_{1\leq k,m\leq d_{\xi}}$$ for a
unitary irreducible representation $\xi\in[\xi]\in\widehat{G}$. The difference
is that in the first case we will have that the eigenvalues of $E$ corresponding
to $H_{j}$'s are all distinct, while in the second case the eigenvalues of the Laplacian
on $G$ for which $H_{j}$'s are the eigenspaces, may coincide.

\begin{defn}\label{inva1}
In view of properties (A) and (C), respectively, an operator $T$ satisfying any of
the equivalent properties (A)--(C) in
Theorem \ref{THM:inv-rem}, will be called an {\em invariant operator}, or
a {\em Fourier multiplier relative to the decomposition
	$\{H_{j}\}_{j\in\N_{0}}$} in \eqref{EQ:sum}.
If the collection $\{H_{j}\}_{j\in\N_{0}}$
is fixed once and for all, we can just say that $T$ is {\em invariant}
or a {\em Fourier multiplier}.

The family of matrices $\sigma$ will be
called the {\em matrix symbol of $T$ relative to the partition $\{H_{j}\}$ and to the
	basis $\{e_{j}^{k}\}$.}
It is an element of the space $\Sigma$ defined by
\begin{equation}\label{EQ:Sigma1}
\Sigma=\{\sigma:\N_{0}\ni\ell\mapsto\sigma(\ell)\in \ce^{d_{\ell}\times d_{\ell}}\}.
\end{equation}
\end{defn}

For $f\in\Hcal$, in the notation of Theorem \ref{THM:inv-rem},
by definition we have
\begin{equation}\label{EQ:ser}
f=\sum_{j=0}^{\infty} \sum_{k=1}^{d_{j}} \widehat{f}(j,k) e_{j}^{k}
\end{equation}
with the convergence of the series in $\Hcal$.
Since $\{e^k_j\}_{j\geq 0}^{1\leq k\leq d_j}$ is a complete orthonormal 
system on  $\Hcal$, for all $f\in \Hcal$ we have the Plancherel formula
\beq \label{EQ:Plancherel}
\|f\|^2_{\Hcal}=\sum\limits_{j=0}^{\infty}\sum\limits_{k=1}^{d_j}|( f,e_j^k)|^2
=  \sum\limits_{j=0}^{\infty}\sum\limits_{k=1}^{d_j}|\widehat{f}(j,k)|^{2}
=\|\widehat{f}\|^{2}_{\ell^2(\N_{0},\Sigma)},
\eq
where we interpret $\widehat{f}\in\Sigma$ as an element of the space
\begin{equation}\label{EQ:aux3}
\ell^2(\N_{0,}\Sigma)=
\{h:\ene_0\rightarrow \prod\limits_d\ce^{d}: h(j)\in \ce^{d_j}\, \mbox{ and }\,\sum\limits_{j=0}^{\infty}\sum\limits_{k=1}^{d_j}|h(j,k)|^2<\infty\}, 
\end{equation}
and where we have written $h(j,k)=h(j)_k$. 
In other words, $\ell^2(\N_{0,}\Sigma)$ is the space of all $h\in\Sigma$ such that
$$
\sum\limits_{j=0}^{\infty}\sum\limits_{k=1}^{d_j}|h(j,k)|^2<\infty.
$$
We endow  $\ell^2(\N_{0},\Sigma)$ with the norm
\begin{equation}\label{EQ:aux4}
\|h\|_{\ell^2(\N_{0,}\Sigma)}:=\left(\sum\limits_{j=0}^{\infty}\sum\limits_{k=1}^{d_j}|h(j,k)|^2\right)^{\half}.
\end{equation}

We note that the matrix symbol $\sigma(\ell)$ depends 
not only on the partition \eqref{EQ:sum} but also
on the choice of the orthonormal basis.
Whenever necessary, we will indicate the dependance of $\sigma$ on the orthonormal 
basis by writing $(\sigma,\{e_j^k\}_{j\geq 0}^{1\leq k\leq d_j} )$ and we also will refer to 
$(\sigma,\{e_j^k\}_{j\geq 0}^{1\leq k\leq d_j} )$ as the {\em symbol} of $T$. 
Throughout this  section the orthonormal basis will be fixed and unless there is some 
risk of confusion the symbols will be denoted simply by $\sigma$.  
In the invariant language,  
we have that the transpose of the symbol,
$\sigma(j)^{\top}=T|_{H_{j}}$ is just the restriction of
$T$ to $H_{j}$, which is well defined in view of the property (A).

We will also sometimes 
write $T_{\sigma}$ to indicate that $T_{\sigma}$ is an operator corresponding to the 
symbol $\sigma $. It is clear from the definition that invariant operators are 
uniquely determined by their symbols. Indeed, if $T=0$ we obtain 
$\sigma=0$  for any choice of an orthonormal basis.  
Moreover, we note that by taking $j=\ell$ in (B) of Theorem \ref{THM:inv-rem} we obtain
the formula for the symbol:
\beq\label{symbinv}
\sigma(j)_{mk}=\widehat{Te_j^k}(j,m),
\eq
for all $1\leq k,m\leq d_j$. The formula (\ref{symbinv}) furnishes 
an explicit formula for the symbol in terms of the operator and the orthonormal basis. 
The definition of Fourier coefficients tells us that for invariant operators
we have 
\beq\label{symbinv2}
\sigma(j)_{mk}=({Te_j^k},e_j^m)_{H}.
\eq
In particular,  for the identity operator $T=I$ we have $\sigma_{I}(j)=I_{d_{j}}$,
where $I_{d_{j}}\in \C^{{d_{j}}\times {d_{j}}}$ is the identity matrix.

Let us now indicate a formula relating symbols 
with respect to different orthonormal bases. 
If $\{e_{\alpha}\}$ and $\{f_{\alpha}\}$ are orthonormal bases of 
$\Hcal$, we consider the unitary operator $U$ determined by 
$U(e_{\alpha})=f_{\alpha}$. Then we have
\[
(Te_{\alpha}, e_{\beta})_{\Hcal}=(UTe_{\alpha}, Ue_{\beta})_{\Hcal}
=(UTU^*Ue_{\alpha}, Ue_{\beta})_{\Hcal}
=(UTU^*f_{\alpha}, f_{\beta})_{\Hcal}.
\]
Thus, if $(\sigma_{T}, \{e_{\alpha}\})$ denotes the symbol of $T$ with respect to the 
orthonormal basis $\{e_{\alpha}\}$ and $(\sigma_{UTU^*}, \{f_{\alpha}\})$ 
denotes the symbol of $UTU^*$ with respect to the orthonormal basis $\{f_{\alpha}\}$ 
we have obtained the relation
\beq\label{difsymb} (\sigma_{T}, \{e_{\alpha}\})=({\sigma_{UTU^*}}, \{f_{\alpha}\}).\eq
Thus, the equivalence relation of bases $\{e_{\alpha}\}\sim  \{f_{\alpha}\}$ given by 
a unitary operator $U$ induces
the equivalence relation on the set $\Sigma$ of symbols given by 
\eqref{difsymb}. In view of this,
we can also think of the symbol being independent of a choice of basis, as an element of the space
$\Sigma/\sim$ with the equivalence relation given by
\eqref{difsymb}.

\section{ Symbols on Hilbert Tensor Products}\label{sec5}
We will now apply the above notion of global symbols to the setting of Hilbert tensor products. First we will see the decomposition of Hilbert tensors as direct sums so that we can apply our global symbols. Secondly, we deduce some  Schatten-von Neumann properties for tensors of invariant operators.  \\

The proof of the next lemma follows from Theorem \ref{prhdes}. 
\begin{lem}\label{lesh2x} Let $\Hi_1, \Hi_2$  be two complex separable Hilbert spaces decomposed as direct sums in the form
\[\Hi_1=\bigoplus\limits_{j=1}^{\infty}\Hi_{1,j}\, ,\,\, \Hi_2=\bigoplus\limits_{k=1}^{\infty}\Hi_{2,k}. \]
Then, the Hilbert space $\Hi_1\otimes\Hi_2$ can be written in the form 
\begin{equation}\label{lempart2h}
    \Hi_1\otimes\Hi_2=\bigoplus\limits_{j,\,k=1}^{\infty}(\Hi_{1,j}\otimes\Hi_{2,k})
    \end{equation}
\end{lem}
As a consequence we obtain formulae for the special case of Hilbert tensor of Fourier multipliers or invariant operators relative to fixed decompositions of the Hilbert factors in the sense of the Definition \ref{inva1}. This theorem shows that the notion of invariance herein considered is well behaved with respect to tensors. First we note that the dimension  of $\Hi_{1,j}\otimes\Hi_{2,k}$ is $d_jd_k$, and a typical element of the tensor of the corresponing orthonormal bases is of the form $e_{1,j}^p\otimes e_{2,k}^q$.
\begin{thm}\label{lesh2x2} Let $\Hi_1, \Hi_2$  be two complex separable Hilbert spaces decomposed as direct sums in the form
\[\Hi_1=\bigoplus\limits_{j=1}^{\infty}\Hi_{1,j}\, ,\,\, \Hi_2=\bigoplus\limits_{k=1}^{\infty}\Hi_{2,k}. \]

Let us consider two invariant operators 
\[\Op(\sigma_1):\Hi_1\rightarrow\Hi_1,\,\, \Op(\sigma_2):\Hi_2\rightarrow\Hi_2,\]
corresponding to symbols $\sigma_1, \sigma_2$ in the sense of Definition \ref{inva1}. Then, the corresponding tensor product of operators $\Op(\sigma_1)\otimes\Op(\sigma_2):\Hi_1\otimes\Hi_2\rightarrow \Hi_1\otimes\Hi_2$, is invariant with respect to the partition \eqref{lempart2h} and  we have
\[ \Op(\sigma_1)\otimes\Op(\sigma_2)=\Op(\sigma_1\otimes\sigma_2),\]
where 
\[(\sigma_1\otimes\sigma_2)(j,k)=\sigma_1(j)\otimes\sigma_2(k).\]
\end{thm}
\begin{proof} The fact that $\Op(\sigma_1)\otimes\Op(\sigma_2)$ is invariant follows by verifying the Condition (A) in  Theorem \ref{THM:inv-rem}. Indeed, since 
\[\Hi_1=\bigoplus\limits_{j=1}^{\infty}\Hi_{1,j}\, ,\,\, \Hi_2=\bigoplus\limits_{k=1}^{\infty}\Hi_{2,k} \]
and $\Op(\sigma_1)(\Hi_{1,j})\subset \Hi_{1,j},\,\, $
$ \Op(\sigma_2)(\Hi_{2,j})\subset \Hi_{2,j} $. We obtain %
\begin{align*}
 \left(\Op(\sigma_1)\otimes\Op(\sigma_2) \right)(\Hi_{1,j}\otimes\Hi_{2,k})&\subset \Op(\sigma_1)(\Hi_{1,j})\otimes \Op(\sigma_2)(\Hi_{2,k}) \\
 &\subset \Hi_{1,j}\otimes\Hi_{2,k}.
\end{align*}%
This shows the invariance of $\Op(\sigma_1)\otimes\Op(\sigma_2).$ In order to obtain the symbol of $\Op(\sigma_1)\otimes\Op(\sigma_2)$, we first observe that for an element $e_{1,j}^p\otimes e_{2,k}^q$ of the orthonormal basis of $\Hi_1\otimes \Hi_2$ we have
\begin{equation}\label{prinrh1}
(\Op(\sigma_1)\otimes\Op(\sigma_2))(e_{1,j}^p\otimes e_{2,k}^q)=\Op(\sigma_1)(e_{1,j}^p)\otimes\Op(\sigma_2)(e_{2,k}^q).
\end{equation}
We now use the identity \eqref{symbinv2} and see that  
\begin{align*}
\left\langle \Op(\sigma_1)(e_{1,j}^p)\otimes\Op(\sigma_2)(e_{2,k}^q), 
 e_{1,j}^r\otimes e_{2,k}^s\right\rangle &=\left\langle \Op(\sigma_1)(e_{1,j}^p), e_{1,j}^r\right\rangle\left\langle \Op(\sigma_2)(e_{2,k}^q), e_{2,k}^s\right\rangle\\
 &=\sigma_1(j)_{(r,p)}\sigma_2(k)_{(s,q)}.
 \end{align*}
By writing $T=\Op(\sigma_1)\otimes\Op(\sigma_2)$, we have that the symbol $\sigma_T$ of $T$, is given by
\[\sigma_T(j,k)_{(r,s),(p,q)}=\sigma_1(j)_{(r,p)}\sigma_2(k)_{(s,q)}.\]
Therefore
\[(\sigma_1\otimes\sigma_2)(j,k)=\sigma_1(j)\otimes\sigma_2(k),\]
completing the proof.
\end{proof}
Theorem \ref{lesh2x2} can of course  be extended to the tensor of finitely many invariant operators. 
\begin{cor} \label{cortenx}  Let $\Hi_1, \Hi_2,\dots, \Hi_n$  be    complex separable Hilbert spaces 
 and assume each one can be decomposed as direct sums in the form
\[\Hi_i=\bigoplus\limits_{j=1}^{\infty}\Hi_{i,j}. \]
Let us consider invariant operators 
\[\Op(\sigma_i):\Hi_i\rightarrow\Hi_i, \]
with corresponding  symbols $\sigma_i$ for $i=1,2,\dots, n$ in the sense of Definition \ref{inva1}. Then, the corresponding tensor product of operators 
\[\bigotimes\limits_{i=1}^{n} \Op(\sigma_i):\bigotimes\limits_{i=1}^{n}\Hi_i\rightarrow \bigotimes\limits_{i=1}^{n}\Hi_i,\] is invariant with respect to the partition
\begin{equation}\label{part29k}
\bigotimes\limits_{i=1}^{n}\Hi_i= \bigoplus\limits_{j_1,j_2,\dots,j_n=1}^{\infty}(\Hi_{1,j_1}\otimes\Hi_{2,j_2}\otimes\cdots\otimes \Hi_{n,j_n}).
\end{equation}

The symbol of the tensor satisfies 
\[ \bigotimes\limits_{i=1}^{n}\Op(\sigma_i)=\Op\left(\bigotimes\limits_{i=1}^{n}\sigma_i\right),\]
where
\[\left(\bigotimes\limits_{i=1}^{n}\sigma_i\right)(j_1,j_2,\dots,j_n)=\bigotimes\limits_{i=1}^{n}\sigma_i(j_i).\]
 \end{cor}
\begin{proof} The fact that \eqref{part29k} holds follows by induction on Lemma \ref{lesh2x}: The invariance of $\bigotimes\limits_{i=1}^{n} \Op(\sigma_i)$ with respect to such partition follows by induction applying Theorem \ref{lesh2x2} as well as the corresponding formula for the symbol. 
\end{proof}
We now give some consequences of Corollary \ref{corsp}, the corollary above and the Theorem  2.5 of \cite{fjdmr:foum} in the setting of Schatten-von Neumann classes.
\begin{cor} \label{invSchatten0} Let  $0<p<\infty$ and   $\Hi_j$  complex separable Hilbert spaces for $j=1, 2, \dots \, n$. If $A_j\in S_p(\Hi_j)$ are invariant  for $j=1, 2, \dots \, n$, then  
\[\bigotimes \limits_{j=1}^{n}A_j\in S_p\left(\bigotimes \limits_{j=1}^{n}\Hi_j\right).\]
Moreover
\[\|\bigotimes \limits_{j=1}^{n}A_j\|_{S_p\left(\bigotimes \limits_{j=1}^{n}\Hi_j\right)}=\prod\limits_{j=1}^{n}\|A_j\|_{S_p(\Hi_j)}=\prod\limits_{j=1}^{n}\left(\sum\limits_{\ell=1}^{\infty}\|\sigma_j(\ell)\|_{S_p(\Hi_{\ell})}^p \right)^{\frac{1}{p}}.\]
In the case of the  trace class $(p=1)$, we additionally have 
\begin{equation}
\Tr\left(\bigotimes \limits_{j=1}^n A_j\right)=\prod\limits_{j=1}^{n}\Tr(A_j)=\prod\limits_{j=1}^{n}\left(\sum\limits_{\ell=1}^{\infty}\Tr(\sigma_j(\ell) )\right).\label{trinf2h}\end{equation}	 
\end{cor}




As an example we derive formulae for Schatten-von Neumann norms for tensors of negative powers $(I-\mathcal{L}_{\SU2})^{-\frac{\alpha}{2}}$ of the Laplacian on $\SU2$. By applying  Corollary \ref{invSchatten0} above and Corollary 4.5 of \cite{dr13:schatten}, we have: %

\begin{cor}\label{COR:su2-Lap}%
Let $0<p<\infty$ and  $\alpha,\, \beta> \frac{3}{p}$ . Then, the tensor 
\[(I-\lapsu2)^{-\frac{\alpha}{2}}\otimes (I-\lapsu2)^{-\frac{\beta}{2}}\]
is invariat and belongs to $S_p(L^2(\SU2\times\SU2))$.\\

Moreover, the norm of the tensor $(I-\lapsu2)^{-\frac{\alpha}{2}}\otimes (I-\lapsu2)^{-\frac{\beta}{2}}$ satisfies
\begin{align*}
\|(I-\lapsu2)^{-\frac{\alpha}{2}}\otimes (I-\lapsu2)^{-\frac{\beta}{2}}\|_{S_p}&=\|(I-\lapsu2)^{-\frac{\alpha}{2}}\|_{S_p}\|(I-\lapsu2)^{-\frac{\beta}{2}}\|_{S_p} 
\end{align*}
\[=\left( \sum\limits_{\ell\in\frac{1}{2}\ene_0}(2\ell+1)^2(1+\ell(\ell+1))^{-\frac{\alpha}{2}p} \sum\limits_{l\in\frac{1}{2}\ene_0}(2l+1)^2(1+l(l+1))^{-\frac{\beta}{2}p}\right)^{\frac{1}{p}}.\]

\end{cor}

\begin{proof} By  Corollary 4.5 of \cite{dr13:schatten} we have that the invariant operators 
$(I-\lapsu2)^{-\frac{\alpha}{2}}$ and $(I-\lapsu2)^{-\frac{\beta}{2}}$ belong to $S_p(L^2(\SU2))$. From Corollary \ref{invSchatten0} we have that the tensor product\\ $(I-\lapsu2)^{-\frac{\alpha}{2}}\otimes (I-\lapsu2)^{-\frac{\beta}{2}}$ is invariant and  belongs to $S_p(L^2(\SU2)\otimes 
 L^2(\SU2))$. \\
 
 Since $L^2(\SU2)\otimes L^2(\SU2)$ is isomorphic to $L^2(\SU2\times\SU2)$, 
 we conclude the proof of the first part.\\

On the other hand, the formula for the Schatten-von Neumann norm of each factor follows from the formulae for the global symbol of  $\lapsu2$ (see \cite{dr13:schatten}).     
\end{proof}
Similar conclusions may be drawn for negative powers of the sub-Laplacian and trace formulae.
Other examples can include a family of `Schr{\"o}dinger operators' on $\SO3$:
\[\mathcal H_{\gamma}=iD_3-\gamma(D_1^2+D_2^2),\]
for a parameter $0<\gamma<\infty,$ and where we give fix three invariant vector fields $D_1, D_2, D_3$ on $\SO3$ corresponding to the derivatives with
respect to the Euler angles. We refer to \cite[Chapter 11]{rt:book} for the explicit formulae.\\

In that case the matrix-symbol of $I+\mathcal H_{\gamma}$ is given by 
\begin{equation}\label{EQ:SO3-schor}
\sigma_{I+\mathcal H_{\gamma}}(\ell)_{mn}=(1+m-\gamma m^2+\gamma\ell(\ell+1))\delta_{mn},\quad m,n\in\mathbb Z,\;
-\ell\leq m,n\leq\ell,
\end{equation}
where as before $\delta_{mn}$ is the Kronecker delta, and we let $m,n$ run from
$-\ell$ to $\ell$ rather than from $0$ to $2\ell+1$.\\

In this last part of the paper we deduce some consequences for a class of periodic pseudodifferential operators, or pseudodifferential operators on the flat torus. We will be mainly focused on those ones with symbols of the form $\sigma(x,j)=a(x)\beta(j)$ defined on $\Tn\times \Zn$. This type of operators has been of recent interest for computational purposes with neural networks  (cf. \cite{pnetw:ko}), where families of operators of this kind play an essential role. Some recent results regarding singular traces for pseudo-differential operators on the flat torus can also be found in  \cite{Piet:trace2}.\\

We will consider a family of symbols $\sigma_m$ of the form  $\sigma_m(x,j)=a_m(x)\beta_m(j)$, where $a_m:\T\rightarrow\er$ is a nonnegative measurable function for $m=1,\dots, N$; $\beta_m:\zet\rightarrow \er$ 
is a nonegative function for $m=1,\dots, N$ such that 
\[ (A1)\,\,\,\|a_m\|_{L^{\infty}(\T)}<\infty,\,\, \sum\limits_{j\in\zet}\beta_m(j)^p<\infty\]
  for all $1<p<\infty$ and for all  $m=1,\dots, N$.\\

In the following theorem, $\Tr_{\omega}$ denotes the Dixmier trace (cf.\cite{Connes}), and we give a sufficent condition for the existence of Dixmier trace for finite tensors of the operators above described.

\begin{thm} Let $A_m$ be pseudodifferential operators with symbol of the form $\sigma_m(x,j)=a_m(x)\beta_m(j)$ for for $m=1,\dots, N$. We assume that they satisfy $(A1)$ and that the limit
\beq\label{dix1} \lim\limits_{p\rightarrow 1^{+}}(p-1)\prod\limits_{m=1}^{N}\|a_m\|_{L^{\infty}(\T)}^p\|\beta_m\|_{\ell^{p}(\zet)}^p\eq
exists. Then, 
\beq \label{dix2}
\lim\limits_{p\rightarrow 1^{+}}(p-1)\Tr\left(\left(\bigotimes_{m=1}^N A_m\right)^p \right) 
\eq
exists, and 
\beq \label{dix3}\Tr_{\omega}\left(\bigotimes_{m=1}^N A_m\right)=\lim\limits_{p\rightarrow 1^{+}}(p-1)\Tr\left(\bigotimes_{m=1}^N A_m^p\right) .\eq
\end{thm}
\begin{proof} We first note that since each $A_m$ is a positive definite operator on $L^{2}(\T)$, then so is 
$\bigotimes_{m=1}^N A_m$ on $\bigotimes_{m=1}^N L^{2}(\T)$, and 
\[\left(\bigotimes_{m=1}^N A_m\right)^p=\bigotimes_{m=1}^N A_m^p,\]
for all $1<p<\infty$. \\

We now note that, $A_m^p$ is a trace class operator for all $1<p<\infty$. Indeed, since $A_m$ is a positive definite operator belonging to the Schatten-von Neumann class $S_p$ as we will see below and 
\[\|A_m^p\|_{S_1}=\|A_m\|_{S_p}^{p}. \]
The operator $A_m$ can be written as a composition of a multiplication operator that is bounded with a Fourier multiplier that is trace class since the symbol of $A_m$ is   $a_m(x)\beta_m(j)$. We observe that
\begin{eqnarray*}
\|A_m\|_{S_p}\leq &\|T_{a_m}\|_{\mathcal{L}(L^2(\T))}\|\beta_m(D)\|_{S_p}\\
 \leq& \|a_m\|_{L^{\infty}(\T)}\|\beta_m(D)\|_{S_p}\\
 =&  \|a_m\|_{L^{\infty}(\T)}\left(\sum\limits_{j\in\zet}\beta_m(j)^p\right)^{\frac{1}{p}}
\end{eqnarray*}
where we have denote by $T_{a_m}$ the multiplication operator associated to $a_m$ and  by $\|\cdot\|_{\mathcal{L}(L^2(\T))}$ the operator norm with respect to the $L^2(\T)$ norm. From this and the assumption in the theorem we get that  $A_m^p$ is a trace class operator for all $1<p<\infty$.

Then, we have 
\begin{eqnarray*} \Tr\left(\left(\bigotimes_{m=1}^N A_m\right)^p\right)=&\Tr\left(\bigotimes_{m=1}^N A_m^p\right)\\ 
    =& \prod\limits_{m=1}^{N}\Tr(A_m^p)\\
    =& \prod\limits_{m=1}^{N}\|A_m\|_{S_p}^{p}\\
    \leq &  \prod\limits_{m=1}^{N} \|a_m\|_{L^{\infty}(\T)}^p\sum\limits_{j\in\zet}\beta_m(j)^p.
\end{eqnarray*}
From this last inequality and \eqref{dix1}, we can conclude  that the limit  \eqref{dix1} exists and 
\eqref{dix3} follows from  the Connes-Moscovici formula (cf. Proposition 4, \cite{Connes}).
\end{proof}

\noindent {\bf{Declarations}}\\

\noindent $\bullet$ Our manuscript has not associated data.\\

\noindent $\bullet$ No conflict of interest/Competing interests



\begin{thebibliography}{LYLW20}



\bibitem[CDGR23]{CGDR:Hi}
D. Cardona, J. Delgado, B. Grajales, M. Ruzhansky.
\newblock {C}ontrol of the Cauchy problem on Hilbert spaces: A global approach via symbol criteria. {\em Comm. Pure Appl. Anal.}, 22, 3295-3329, 2023

\bibitem[CDC20]{cdc20} D. Cardona, C. Del Corral. The Dixmier trace and the non-commutative residue for multipliers on compact manifolds. In: Georgiev V., Ozawa T., Ruzhansky M., Wirth J. (eds) Advances in Harmonic Analysis and Partial Differential Equations. Trends in Mathematics. Birkh\"auser, Cham, (2020).

\bibitem[CKC20]{ckc20} Cardona, D., Kumar, V., Del Corral, C. Dixmier traces for discrete pseudo-differential operators, J. Pseudo-Differ. Oper. Appl., Vol. 11, 647--656, 2020.

\bibitem[CC20]{cc20}  Cardona, D., Del Corral, C. The Dixmier trace and the Wodzicki residue for pseudo-differential operators on compact manifolds, Rev. Integr. Temas. Mat., Vol. 38 (1), 67-79, 2020.


\bibitem[CDR21]{anh:cdr}
M.~Chatzakou, J.~Delgado, and M.~Ruzhansky.
\newblock {O}n a class of anharmonic oscillators.
\newblock {\em J. Math. Pures Appl.}, 153:1--29, 2021.

\bibitem[CO94]{Connes} Connes, A. 
\newblock {N}oncommutative Geometry.
\newblock {\em Academic Press.}, 1994.



\bibitem[DR14a]{dr13a:nuclp}
J.~Delgado and M.~Ruzhansky.
\newblock ${L}^p$-nuclearity, traces, and {G}rothendieck-{L}idskii formula on
  compact {L}ie groups.
\newblock {\em J. Math. Pures Appl.}, 102:153--172, 2014.

\bibitem[DR14b]{dr:suffkernel}
J.~Delgado and M.~Ruzhansky.
\newblock Schatten classes on compact manifolds: {K}ernel conditions.
\newblock {\em {J}. {F}unct. {A}nal.}, 267:772--798, 2014.

\bibitem[DR17]{dr13:schatten}
J.~Delgado and M.~Ruzhansky.
\newblock {S}chatten classes and traces on compact groups.
\newblock {\em Math. Res. Letters}, 24:979--1003, 2017.

\bibitem[DR18]{fjdmr:foum}
J.~Delgado and M.~Ruzhansky.
\newblock {F}ourier multipliers, symbols and nuclearity on compact manifolds.
\newblock {\em {J}. {A}nal. {M}ath.}, 135(2):757$–$800, 2018.

\bibitem[DR21]{dr:intsc}
J.~Delgado and M.~Ruzhansky.
\newblock {S}chatten-von {N}eumann classes of integral operators.
\newblock {\em J. Math. Pures Appl.}, 154: 1--29, 2021.

\bibitem[GG18]{gg:qq1}
F.~Genovese and S.~Gogioso.
\newblock {Q}uantum field theory in categorical quantum mechanics. proceedings
  14th international conference on quantum physics and logic.
\newblock {\em Electron. Proc. Theor. Comput. Sci.}, (266):349--366, 2018.

\bibitem[GG19]{gg:qq2}
F.~Genovese and S.~Gogioso.
\newblock {Q}uantum field theory in categorical quantum mechanics. proceedings
  15th international conference on quantum physics and logic.
\newblock {\em Electron. Proc. Theor. Comput. Sci.}, (287):163--177, 2019.

\bibitem[GS18]{mgas:sp}
M.~Gessner and A.~Smerzi.
\newblock {S}tatistical speed of quantum states: {G}eneralized quantum {F}isher
  information and {S}chatten speed.
\newblock {\em {P}hysical {R}eview {A}.}, 97:022109, 2018.

\bibitem[Gui72]{gui:hi}
A.~Guichardet.
\newblock {\em {S}ymmetric {H}ilbert Spaces and Related Topics. {L}ecture notes
  in mathematics}, volume 261.
\newblock Springer-Verlag Berlin Heidelberg, 1972.%

\bibitem[Gui66]{gui:1} A. Guichardet.  {\em  Produits tensoriels infinis et representations des relations d'anticommutation.} (French) Ann. Sci. Ecole Norm. Sup. (3) 83, 1–-52, 1966.%

\bibitem[LYLW20]{schit:cc1}
J.~Liu, H.~Yuan, X.-M. Lu, and X.~Wang.
\newblock {Q}uantum {F}isher information matrix and multiparameter estimation.
\newblock {\em J. Phys. A.}, 53(2), 2020.

\bibitem[MuN36]{Murray} F. J. Murray and J. von Neumann. On rings of operators. {\em Ann. of Math.} (2) 37 (1936), no. 1, 116–229.


\bibitem[Par92]{Pah:book}
R.~K. Parthasarathy.
\newblock {\em {A}n introduction to quantum stochastic calculus. {M}onographs
  in {M}athematics}, volume~85.
\newblock Birk\"auser Verlag, Basel, 1992.

\bibitem[Piet15]{Piet:trace2}
Pietsch, A. Traces and Residues of Pseudo-Differential Operators on the Torus. Integr. Equ. Oper. Theory 83, 1–23, 2015.

\bibitem[Pru81]{prug:b1}
E.~Prugovecki.
\newblock {\em {Q}uantum mechanics in {H}ilbert Spaces: second edition}.
\newblock Academic Press, New York, 1981.

\bibitem[RS80]{r-s:vol1}
M.~Reed and B.~Simon.
\newblock {\em Methods of modern mathematical physics. {I}}.
\newblock Academic Press, Inc. [Harcourt Brace Jovanovich, Publishers], New
  York, second edition, 1980.
\newblock Functional analysis.

\bibitem[RT10a]{rt:book}
M.~Ruzhansky and V.~Turunen.
\newblock {\em Pseudo-differential operators and symmetries. Background
  analysis and advanced topics}, volume~2 of {\em Pseudo-Differential
  Operators. Theory and Applications}.
\newblock Birkh{\"a}user Verlag, Basel, 2010.


\bibitem[SLH2024]{pnetw:ko} Jin Young Shin, Jae Yong Lee and Hyung Ju Hwang. Pseudo-Differential Neural Operator: Generalized Fourier
{N}eural Operator for Learning Solution Operators of Partial
Differential Equations. {\em Transactions on Machine Learning Research}, 3, 2024


\bibitem[Sch43]{Sch}
  R. Schatten. On the direct product of Banach spaces. {\em Trans. Amer. Math. Soc.} 53 (1943), 195–217

\bibitem[Sch46]{Sch1}
  R. Schatten. The cross-space of linear transformations. {\em Ann. of Math.} (2) 47 (1946), 73–84.

   
\bibitem[SchNe46]{SN}
    R. Schatten and J. von Neumann. The cross-space of linear transformations. II. {\em Ann. of Math.} (2) 47 (1946), 608–630. 

\bibitem[Sch70]{Sch2}
 R. Schatten. Norm ideals of completely continuous operators. Second printing. Ergebnisse der Mathematik und ihrer Grenzgebiete, Band 27 Springer-Verlag, Berlin-New York 1970 vii+81 pp.

\bibitem[Ne39]{Ne}
 J. von Neumann. On infinite direct products. {\em Compositio Math}. 6 (1939), 1–77.

\bibitem[Wea01]{weav:book}
N.~Weaver.
\newblock {\em {M}athematical Quantization}.
\newblock Chapman \& Hall/ Crc, 2001.


\bibitem[Wei80]{weid:b1}
J.~Weidmann.
\newblock {\em {L}inear Operators in Hilbert Spaces}, volume~68.
\newblock Springer-Verlag, New York, 1980.

\end{thebibliography}

\end{document}